\documentclass[12pt]{amsart}

\usepackage[english]{babel}
\usepackage[latin1]{inputenc}
\usepackage{amsmath,amssymb, color,enumerate,amssymb,wasysym,mathrsfs,dsfont}
 \usepackage{mathptmx}
 \usepackage[all,cmtip]{xy}
\usepackage{xcolor}
\usepackage{soul}
 \DeclareMathAlphabet{\mathcal}{OMS}{cmsy}{m}{n}

\DeclareSymbolFont{rsfscript}{OMS}{rsfs}{m}{n}
\DeclareSymbolFontAlphabet{\mathrsfs}{rsfscript}

\DeclareSymbolFont{AMSb}{U}{msb}{m}{n}
\DeclareSymbolFontAlphabet{\mathbb}{AMSb}

\DeclareSymbolFont{eufrak}{U}{euf}{m}{n}
\DeclareSymbolFontAlphabet{\gothic}{eufrak}

\newcommand\Imm{\operatorname{Im}}

\DeclareMathOperator{\Triv}{\bf Triv}

\newcommand{\Rg}{\mathsf{RGrp}}
\newcommand{\Rzs}{\mathsf{Rzs}}
\newcommand{\Set}{\mathsf{Set}}
\newcommand{\Sym}{\operatorname{Sym}}
\newcommand{\Aut}{\operatorname{Aut}}
\newcommand{\End}{\operatorname{End}}

\newcommand{\Cal}[1]{{\mathcal #1}}
\newcommand{\Grp}{\mathsf{Grp}}
\newcommand{\DiSGrp}{\mathsf{DiSGrp}}

\newcommand{\lGrp}{\mbox{\rm -{\bf Grp}}}

\newcommand{\id}{\operatorname{id}}

\newcommand{\RGrp}{\mathsf{RGrp}}
\DeclareMathOperator{\res}{res}
\DeclareMathOperator{\ext}{ext}

\newtheorem{theorem}{Theorem}[section]
\newtheorem{proposition}[theorem]{Proposition}
\newtheorem{lemma}[theorem]{Lemma}
\newtheorem{corollary}[theorem]{Corollary}
\theoremstyle{definition}
\newtheorem{definition}[theorem]{Definition}
\newtheorem{remark}[theorem]{Remark}

\newtheorem{Ex}[theorem]{Example}

\newcommand{\SGrp}{\mathsf{SGrp}}
\newcommand{\DSGp}{\mathsf{DSGp}}
\newcommand{\Hom}{\operatorname{Hom}}

\begin{document}
  \title{A pretorsion theory for right groups}
     \author{Alberto Facchini}
\address[Alberto Facchini]{Dipartimento di Matematica ``Tullio Levi-Civita'',\linebreak Universit\`a di 
Padova, 35121 Padova, Italy}
 \email{facchini@math.unipd.it}
\thanks{The second author was partially supported by GNSAGA, by the research project PIACERI ``ACIVA - Anelli commutativi, loro
	ideali e varietà algebriche'' and by the research project PRIN 2022 ``Unirationality,
	Hilbert schemes, and singularities''.}
\author{Carmelo Antonio Finocchiaro}
\address[Carmelo Antonio Finocchiano]{Dipartimento di Matematica e Informatica, Universit\`a\ di Catania, Citt\`a\ Universitaria, viale Andrea Doria 6, 95125 Catania, Italy}
\thanks{}
\email{cafinocchiaro@unict.it}


      \begin{abstract} Let $S$ be a right group. Then there exist two congruences $\sim$ and $\equiv$ on $S$ such that $S$ is the product of its quotient semigroups $S/{\sim}$ and $S/{\equiv}$, where $S/{\sim}$ is a group and $S/{\equiv}$ is a right zero semigroup. If $E$ is the set of all idempotents of $S$ and we fix an element $e_0\in E$, then the pointed right group $(S,e_0)$ is the coproduct of its pointed subsemigroups $(Se_0,e_0)$ and $(E,e_0)$ in the category of pointed right groups. In general, there is a pretorsion theory in the category of right groups in which the torsion objects are right zero semigroups and the torsion-free objects are groups.
  \end{abstract}

    \maketitle


\section{Introduction} 

Right groups form a widely studied class of semigroups \cite[Section~1.11]{CPI}. Every right group $S$ is, up to isomorphism, the product of a non-empty right zero semigroup $E$
and a group $G$. In applications, it is often convenient to also consider {\em pointed} right groups, instead of right groups. This occurs, for instance, in the study of digroups (\cite[Section~4]{Kinyon} and  \cite{FFD}). A {\em pointed right group} is a pair $(S,e_0)$, where 
$S$ is a right group and $e_0$ is an idempotent element of $S$. 

The category of right groups and that of pointed right groups exhibit rather different behavior. Some of the differences are obvious. For example, the category $\RGrp$ of right groups has no initial object, whereas the category $\RGrp_*$ of pointed right groups has a null object. The main difference between the two categories, however, lies in the representation of an object $S$
as a product of a right zero semigroup $E$ and a group $G$. In the category $\RGrp_*$ of pointed right groups, this representation $S\cong E\times G$ is simultaneously a product and a coproduct. By contrast, in the category $\RGrp$ of right groups we are faced with a pretorsion theory \cite{FF, fa-fi-gr}: for every right group $S$
there is a pre-exact sequence $E(S) \xrightarrow{i} S \xrightarrow{\pi} S/{\sim}$
in $\RGrp$, where $E(S) $ is the subsemigroup of all idempotent elements of $S$, 
$\sim$ is the smallest congruence on $S$ in which all elements of $E(S)$ are pair-wise congruent, and the quotient semigroup $S/{\sim}$ is a group (Section~\ref{pretorsion}). Thus, this is a pretorsion theory in the category $\RGrp$ of right groups in which the pretorsion objects are the nonempty right zero semigroups, and the torsion-free objects are the groups.
Note that this pretorsion theory is not a torsion theory in the category $\RGrp$ of right groups, since on the one hand the category $\RGrp$  has no null object but only terminal objects, and on the other hand because in this pretorsion theory there are several trivial morphisms between two objects $S$ and 
$S'$, one for each idempotent element of $S'$.  That is, there is exactly one trivial morphism from 
$S$  to $S'$ for each idempotent element $e'$ of $S'$, namely the morphism constantly equal to $e'$.
The trivial morphisms $S\to S$
 induce modulo $\sim$ all and only the morphisms $S/{\sim}\to S$  in the category $\RGrp$ that are right inverses of the canonical projection $\pi\colon S\to S/{\sim}$. In the category $\RGrp_*$ of pointed right groups, on the contrary, the canonical projection $\pi\colon S\to S/{\sim}$ has a unique right inverse.
 
The category $\RGrp_*$ of pointed right groups is equivalent to the product category of the category $\Set_*$ of pointed sets and the category $\Grp$ of groups (Theorem~\ref{equiv2}), and the category $\RGrp$ of right groups is equivalent to the product category of the category $\Set_{\ne\emptyset}$ of non-empty sets and the category $\Grp$ (Theorem~\ref{equiv1}). 

Right groups and pointed right groups form varieties in the sense of Universal Algebra, of signature (2,2) and $(2,0,1)$  respectively (Section~\ref{Crg}).
Trivially, the classes of semigroups that are groups or right groups do not form subvarieties of the variety of signature $(2)$ of semigroups, since multiplicatively closed subsets of groups or right groups are not, respectively, groups or right groups. 

\section{Preliminary notions and terminology}

\subsection{Direct-product decompositions of semigroups}

For any pair $S,T$ of semigroups, the {\em external direct product} $S\times T$ of $S$ and $T$ is the cartesian product endowed with the componentwise multiplication. As far as internal direct-product decompositions are concerned, notice that in category theory, the fact that $S\cong A\times B$ ($S$ is isomorphic to the product of two objects $A$ and $B$) is equivalent to the existence of two morphisms $\varphi\colon S\to A$ and $\psi\colon S\to B$ whose product $\varphi\times\psi\colon S\to A\times B$ is an isomorphism. In this case, $\varphi$ and $\psi$ are necessarily epimorphisms. For semigroups, this means that we must not talk of direct-product decompositions of a semigroup $S$ as a direct product of two subsemigroups $A$ and $B$ of $S$, but of a direct product of $S$ as a direct-product decomposition of two {\em quotients} $A$ and $B$ of $S$. Equivalently, a direct product of $S$ corresponds to (is) a pair $(\sim,\equiv)$ of congruences of $S$. 

\begin{remark} {\rm This explains why when we deal with direct-product decompositions of a group $G$, we don't deal with a product-decomposition of $G$ as a direct product of two {\em subgroups} of $G$, but as a direct product of two {\em normal} subgroups of $G$. Similarly, when we deal with direct-product decompositions of a right group $S$ in Theorem~\ref{1.1}, we will see that $S$ is the internal direct product of two quotients $E$ and $G$ of $S$ and we will have two projections $S\to E$ and $S\to G$.

Another example is given by the variety of rings with identity. If we construct the external direct product $R\times S$ of two rings $R$ and $S$ with identity, then $R$ and $S$ are homomorphic images of $R\times S$, they are not subrings of $R\times S$, because the identities are different.}\end{remark}

Clearly, the set of all congruences of a semigroup $S$ is a bounded lattice under inclusion with least element the identity congruence $=$ and greatest element the trivial congruence $\omega$.

Now if $S$ is any set and $\sim,\equiv$ are two relations on $S$, it is possible to define their {\em composite relation} $\sim\circ\equiv$, that is $$\{\,(a,b)\in S\times S\mid\ {\mbox{\rm there exists }}c\in S\ {\mbox{\rm with }}a\sim c\ {\mbox{\rm and }}c\equiv b\,\}.$$ 

The following fact is straightforward (see also  \cite[Theorem 5.9]{Burris}).

\begin{lemma} Let $S$ be a semigroup and $\sim,\ \equiv$ be two congruences on $S$. If $\sim$ and $\equiv$ are permutable (that is, $\sim\circ\equiv$ and $\equiv\circ\sim$ coincide), then the least upper bound $\sim\vee\equiv$ of $\sim$ and $\equiv$ in the lattice of all congruences of $S$ coincides with ${\sim}\circ{\equiv}$.\end{lemma}

The standard necessary and sufficient properties for internal direct-prod\-uct decompositions of a group $G$ as a direct product of two normal subgroups $A$ and $B$ (that the intersection $A\cap B$ is trivial and that $AB=G$) for a semigroup $S$ become:

\begin{definition}\label{2.2} {\rm Let $S$ be a semigroup and $\sim,\equiv$ be two congruences on $S$. We say that $S=(S/{\sim})\times (S/{\equiv})$ is the {\em direct product} of its homomorphic images $S/\!\sim$ and $S/\equiv$ if $\sim$ and $\equiv$ are complementary permutable conguences, that is $\sim\circ\equiv$ is equal to $\equiv\circ\sim$ and is equal to the trivial congruence $\omega$, and $\sim\wedge\equiv$ is the identity congruence $=$.}\end{definition}

See \cite[Theorem 5.9]{Burris}. The fact that the two congruences $\sim$ and $\equiv$ must be permutable in Definition \ref{2.2} follows from that fact that if $A$ and $B$ are semigroups, then the external direct products $A\times B$ and $B\times A$ are isomorphic via the switch $s\colon A\times B\to B\times A$, so that there is an isomorphism $\varphi\colon S\to A\times B$ if and only if there is an isomorphism $\psi=e\varphi\colon S\to B\times A$.
Recall that semigroups are not congruence permutable in general, while groups and rings are (\cite[Examples on p.~87]{Burris} and \cite{Ham75}). The condition ``$\sim\circ\equiv$ is equal to the trivial congruence $\omega$'' in Definition \ref{2.2} is equivalent to the fact that the product morphism $$S\to (S/{\sim})\times (S/{\equiv})$$ is a surjective mapping, and the condition ``$\sim\wedge\equiv$ is the identity $=$'' is equivalent to the fact that the product morphism $S\to S/\!\sim\!\times \,S/\!\!\equiv$ is an injective mapping. Of course, Definition~\ref{2.2} applies to any algebra in the sense of Universal Algebra and can be extended from the case of two congruences to any finite number of congruences. 

\begin{definition}\label{mmm}{\rm Let $A$ be any algebra and $\sim_1,\dots,\sim_n$ be $n$ congruences on $A$. Then $A=A/\!\sim_1\times \dots\times A/\!\sim_n$ is the {\em direct product} of its homomorphic images $A/\!\sim_1, \dots, A/\!\sim_n$ if $\sim_1,\dots,\sim_n$ are pairwise permutable conguences that are {\em coindependent}, that is $\sim_i\vee\left(\wedge_{j\ne i}\sim_j\right)$ is equal to the trivial congruence $\omega$ on $A$ for every $i=1,\dots,n$, and $\wedge_{i=1}^n\sim_i$ is the identity congruence~$=$.}\end{definition}

For related results, see \cite[Proposition 6.4]{Rogelio}, \cite[Sections~2.6 and~2.8]{libro1} and \cite[p.~94]{libro2}. Also, compare Definition~\ref{mmm} with the notion of semidirect-product decomposition of an algebra $A$ studied in \cite{David}. Given any algebra $A$, any congruence $\sim$ on $A$ and any subalgebra $B$ of $A$, then $A={\sim}\rtimes B$ is the {\em semidirect product} of the congruence $\sim$ and the subalgebra $B$ if the composite morphism $\pi\iota_B\colon B\to A/{\sim}$ is an isomorphism, where $\iota_B\colon B\to A$ is the inclusion and $\pi\colon A\to A/{\sim}$ is the canonical projection.

\subsection{Right groups}\label{RG}

We will make use of the notation, the terminology and several results in \cite[Section~4]{Kinyon} and \cite{CPI}. 
A semigroup $(S,\cdot)$ is a {\em right zero semigroup} \cite[p.~4]{CPI} if $a\cdot b=b$ for all $a,b\in S$. For these semigroups we will usually write the operation $\cdot$ as $\pi_2$, because it corresponds to the second canonical projection $\pi_2\colon S\times S\to S$. Thus $a\,\pi_2\, b=b$. Similarly, $a\,\pi_1\, b=~a$. Hence right zero semigroups are those of the form $(S,\pi_2)$ for some set $S$. The full subcategory of the category of semigroups whose objects are all right zero semigroups is clearly isomorphic to the category $\Set$ of sets, because every mapping between two right zero semigroups is a semigroup morphism.

\bigskip

For an arbitrary semigroup $S$, let $L$ be the set of all left identities of $S$, that is, $L:=\{\,e\in S\mid ex=x$ for every $x\in S\,\}$. Then $L$ is a subsemigroup of $S$,  the operation induced on it by the operation $\cdot$ of $S$ is the operation $\pi_2$, and we have an embedding $(L,\pi_2)\to (S,\cdot)$. For every $e\in L$, we can consider the centralizer $C_S(e):=\{\, x\in S\mid ex=xe\,\}$ of $e$ in $S$. Then each $C_S(e)$ is also a subsemigroup of $S$, which contains $e$ as a two-sided identity, so that $C_S(e)$ is a monoid. Moreover, each $C_S(e)$ is a left ideal for $S$, that is $yC_S(e)\subseteq C_S(e)$ for every $y\in S$. If $S$ is either left cancellative or right cancellative, then the monoids $C_S(e)$ are pair-wise disjoint, because if $x\in C_S(e)\cap C_S(e')$, then $ex=xe=x$ and $e'x=xe'=x$, so $e=e'$ by one-sided cancellativity. The union of all these left ideals $C_S(e)$ is a  left ideal of $S$, hence in particular a subsemigroup of $S$. This proves that if the union of all these left ideals $C_S(e)$ is the semigroup $S$ itself and $S$ is either left cancellative or right cancellative, then the monoids $C_S(e)$ form a partition of $S$, and the embedding $(L,\pi_2)\to (S,\cdot)$ has a left inverse $(S,\cdot)\to (L,\pi_2)$. This left inverse of the embedding is the semigroup morphism that maps all elements of the block $C_S(e)$ of the partition to $e$. The class of right groups, which we will introduce with Theorem~\ref{1.1}, is a natural class of left cancellative semigroups $S$ that have these properties, that is, such that $S$ is the union of all centralizers of the left identities of $S$.

\bigskip

Given a semigroup $(S,\cdot)$, we will denote by $E(S,\cdot)$ the set of all idempotents of $S$. Recall the following important lemma. 

\begin{lemma}\label{1.2} {\rm \cite[Lemma~1.26]{CPI}}  Every idempotent of a right simple semigroup $S$ is a left identity for $S$.
\end{lemma}

We will now collect several equivalent conditions for a semigroup to be a right group, most of them well known, and we will provide a sketch of the proof for convenience of the reader. We will emphasize some machinery regarding the algebraic-theoretic features of right groups that will be helpful in the rest of the paper.

\begin{theorem}\label{1.1}  {\rm \cite[Section~1.11, Theorem~1.27]{CPI}} The following assertions on a semigroup $S\ne\emptyset$ are equivalent:

{\rm (a)} $S$ is right simple (that is, $aS=S$ for all $a\in S$) and left cancellative.

{\rm (b)} For every $a,b\in S$ there exists a unique element $x\in S$ such that $ax=b$.

{\rm (c)} $S$ is right simple and contains an idempotent.

{\rm (d)} $S$ is isomorphic to the external direct product of a non-empty right zero semigroup $E$ and a group $G$.

{\rm (e)} There exists an element $e\in S$ such that: {\rm (1)}  $e$ is a left identity  for $S$, and {\rm (2)} every element of $S$ has a right inverse with respect to $e$.

{\rm (f)} {\rm (1)} $S$ has a left identity, and {\rm (2)} for every left identity $e$ of $S$ and every element $a\in S$, $a$ has a right inverse with respect to $e$.\end{theorem}

The semigroups satisfying the previous equivalent conditions are called {\em right groups.} 

\begin{proof} (a)${}\Leftrightarrow{}$(b)${}\Leftrightarrow{}$(c)${}\Leftarrow{}$(d).	See  \cite[Section 1.11, Theorem 1.27]{CPI}.

(c)${}\Rightarrow{}$(d).
	This immediately follows from Lemma \ref{1.2}. Let $E$ be the set of all idempotents of $S$. By Lemma~\ref{1.2} every element of $E$ is a left identity for $S$, that is, $ea=a$ for every $e\in E$ and every $a\in S$. In particular, $E$ is a non-empty subsemigroup of $S$ and is a right zero semigroup. Then one fixes any element $e_0\in E$ and takes for $G$ the left  ideal $Se_0$, which turns out to be a group with identity $e_0$. Then it is easy to check that the mapping $G\times E\to S$ defined by $(a,e)\mapsto ae$ for every $(a,e)\in G\times E$ is a semigroup isomorphism between the external direct product $G\times E$ and the semigroup $S$. That is, every $Se_0$ is a complement of $E$ in $S$. This proves (d).
	
(c)${}\Rightarrow{}$(f).  Assume that (c) holds for the semigroup $S$. By condition (c) and \cite[Lemma 1.26]{CPI}, $S$ has a left identity. Suppose now that $f$ is any left identity of $S$ and let $a\in S$. By condition (b), the equation $ax=f$ has a unique solution and, by definition, such a solution is a right inverse of $a$ with respect to $f$. This proves that (c)${}\Rightarrow{}$(f).  

(f)${}\Rightarrow{}$(e) is trivial. 

(e)${}\Rightarrow{}$(c). Let us suppose that (e) holds, let $e$ be a left identity of $S$ and, for every $a\in S$, let $a^{-1}\in S$ denote a right inverse of $a$ with respect to $e$. Given elements $a,s\in S$, we have $s=es=aa^{-1}s$. This proves that $aS=S$ for all $a\in S$, that is, $S$ is right simple. Furthermore, $e$ is trivially an idempotent of $S$. The conclusion is now clear. 
\end{proof}

\begin{remark}\label{sim}
{\em The projection $\pi_G\colon S\to G$, its kernel $\sim$, and the quotient group $S/{\sim}$. }{\rm In Theorem~\ref{1.1}(d), the
projections $\pi_G\colon S\to G$ and $\pi_E\colon S\to E$ are defined by $x\mapsto (x^{-1})^{-1}$ and $x\mapsto  (x^{-1})x$ respectively. Here $x^{-1}$ denotes the right inverse of $x\in S$ with respect to a fixed element $e_0\in E$ (i.e., $x^{-1}$ is the unique element in $S$ such that $x(x^{-1})=e_0$.) By associativity and considering $x(x^{-1})(x^{-1})^{-1}$, one sees that $e_0(x^{-1})^{-1}=xe_0$. Since all elements of $E$ are left identities in $S$, it follows that $(x^{-1})^{-1}=xe_0$. The projection $\pi_G$ is a semigroup morphism, because, for every $x,y\in S$, we  have that $xe_0ye_0=xye_0$. The kernel of $\pi_G$ is the congruence $\sim$ on $S$ defined, for every $x,y\in S$, by $x\sim y$ if $  xe_0=ye_0$. Notice that this congruence $\sim$ does not depend on the choice of the idempotent element $e_0\in E$, because for any other $f\in E$ one has that, for all $x,y\in S$, $xe_0=ye_0$ if and only if $xf=yf$ (because multiplying $xe_0=ye_0$ by $f$ on the right we get that $xe_0f=ye_0f$, that is $xf=yf$; similarly, multiplying by $e_0$ on the right, $xf=yf$ implies $xe_0=ye_0$).  The congruence class modulo $\sim$ of every element  $x\in S$ is $xE$. }

\smallskip

{\em The projection $\pi_E\colon S\to E$, its kernel $\equiv$, and the quotient right zero semigroup $S/{\equiv}$. }{\rm As we have said in the previous paragraph, the
projection $\pi_E\colon S\to E$ is defined by $x\mapsto  (x^{-1})x$. The image $\pi_E(x)=(x^{-1})x$ of $x\in S$ is an idempotent element of $S$, because $((x^{-1})x)^2=x^{-1}xx^{-1}x=x^{-1}e_0x=x^{-1}x$. We have that $\pi_E(x)=(x^{-1})x$ is the unique element $f\in E$ such that $xf=x$. (This characterization also immediately shows that $f=\pi_E(s)$ belongs to $E$, because $sff=sf=s$. It also shows that the projection $\pi_E$, like the congruence $\sim$, does not depend on the choice of the fixed element $e_0$.) The projection $\pi_E$ is a semigroup morphism, because, for every $x,y\in S$, if $f$ is the unique idempotent such that $yf=y$, then also $xyf=xy$, so that $\pi_E(xy)=f=\pi_E(x)\pi_E(y)$. The kernel of the projection  $S\to E$ is the congruence relation $\equiv$ on $S$ for which the quotient semigroup $S/{\equiv}\,$ is $\{\,Se\mid e\in E\,\}$. Notice that $Ge=Se_0e=Se$ for every $ e\in E$. Therefore 
every right group $S$ has a partition $\{\,Se\mid e\in E\,\}$, where every block $Se$ of the partition is a group with identity $e$, so that $E$ is the set of all the identities of these groups $Se$. Moreover all the groups $Se$ are pair-wise isomorphic via the group isomorphism $r_f\colon Se\to Sf$, $r_f(x)=xf$ for all $e,f\in E$, given by right multiplication by $f$.

\smallskip

Now that we have the two congruences $\sim$ and $\equiv$ on any right group $S$, it is clear that the product decomposition of $S$ corresponding to the pair $(\sim,\equiv)$ in the sense of Definition~\ref{2.2} is the product decomposition of $S$ as a right zero semigroup $E$ and a group $G$ of Theorem~\ref{1.1}.(d)}\end{remark}

\begin{Ex} {\rm Let us give an example concerning the equivalent statements of Theorem~\ref{1.1}. Let $S$ be any nonempty set, and consider the semigroup $(S,\pi_2)$, which is obviously a right zero semigroup. Then $S$ is right simple, because $a\,\pi_2\,S=S$ for all $a\in S$; and $S$ is left cancellative, because $a\,\pi_2\,b=a\,\pi_2\,c$ is $b=c$. For every $a,b\in S$ there exists a unique element $x\in S$ such that $ax=b$, it is $b$. Every element of $S$ is
 idempotent. Every element of $S$ is a left identity. For any two elements $e,a\in S$, the right inverse of $a$ with respect to $e$ is $e$. }\end{Ex}
 
\section{Categories of right groups}\label{Crg}

The {\em category of group actions on sets} has, as objects, all triplets  $(G,X, \varphi)$, where $G$ is a group, $X$ is a set, and $\varphi\colon G\to \Aut_{\Set}(X)$, $g\mapsto \varphi_g$, is a group morphism. Clearly $\Aut_{\Set}(X)$ is the symmetric group $\Sym_X$. A morphism
$(G,X,  \varphi)\to (H,Y,\psi)$ is a pair $(\Phi, f)$, where $\Phi\colon G\to H$ is a group morphism, $f\colon X\to Y$ is a mapping, and $\psi_{\Phi(g)}(f(x))=f(\varphi_g(x))$ for every $g\in G$ and $x\in X$. It is immediately seen that if $(\Phi,f):(G,X, \varphi)\to (H,Y,\psi)$ is a morphism in the category of group actions on sets, then $(\Phi,f)$ is an isomorphism if and only if $\Phi$ is an isomorphism of groups and $f$ is a bijection. \\
If $\mathds{1}:G\to\Aut_{\Set}(X)$ is the trivial group morphism, we will say that $(G,X,\mathds{1})$ is the trivial group action of $G$ on $X$.

Similarly, the {\em category of group actions on pointed sets} has, as objects, all 4-tuples  $(G,X, x_0, \varphi)$, where $G$ is a group, $X$ is a set, $x_0$ is a fixed element of $X$, (so that $(X,x_0)$ is a pointed set), and $\varphi\colon G\to \Aut_{\Set_*}(X,x_0)$ is a group morphism. Here $\Aut_{\Set_*}(X,x_0)$ is simply the group of all permutations of $X$ that fix $x_0$, that is, all bijections $f\colon X\to X$ such that $f(x_0)=x_0$. A morphism
$(G,X, x_0, \varphi)\to (H,Y,y_0,\psi)$ is a pair $(\Phi, f)$, where $\Phi\colon G\to H$ is a group morphism, $f\colon X\to Y$ is a mapping, $f(x_0)=y_0$, and $\psi_{\Phi(g)}(f(x))=f(\varphi_g(x))$ for every $g\in G$ and $x\in X$. 

The {\em category of pointed right groups} has as objects the pairs $(S,e_0)$, where $S$ is a right group and  $e_0$ is an idempotent element of $S$, and as morphisms $f\colon (S,e_0)\to (S', e'_0)$ all semigroup morphisms $f\colon S\to S'$ such that $f(e_0)=e'_0$.

Notice that {\em pointed} right groups $(S,\cdot, e_0, {}^{-1})$ do form a variety of algebras, where $\cdot$ is binary and associative, $e_0$ is $0$-ary, ${}^{-1}$ is unary, $e_0x=x$ for every $x\in S$, and $xx^{-1}=e_0$  for every $x\in S$ (Theorem~\ref{1.1}(e)). 
Also
right groups, viewed as algebras $(S,\cdot,\backslash)$ with two binary operations~$\cdot$ and $\backslash$ and the three identities $x\cdot(y\cdot z)=(x\cdot y)\cdot z$, \ $x\cdot (x\backslash y)=y$ and $x\backslash(x\cdot y)=y$,  form a variety in the sense of Universal Algebra. 

\begin{proposition}\label{coproduct} If $S$ is a right group, $e\in E:=E(S)$, $\iota_e\colon Se\to S$ and $\iota_E\colon E\to S$ are the inclusions, $r_e\colon S\to Se$ is defined by $r_e(s)=se$ for every $s\in S$, and $\pi_E\colon S\to E$ is defined by ``$\pi_E(s)$ is the unique element of $S$ such that $s\pi_E(s)=s$'', then the decomposition $S=Se \times E$ shows that: 

{\rm (1)} $(S, r_e,\pi_E)$ is the product of $Se$ and $E$ in the category of semigroups.

{\rm (2)} $((S,e),\iota_e,\iota_E)$ is the coproduct of $(Se,e)$ and $(E,e)$ in the category of pointed right groups.

{\rm (3)} $(S,\iota_e,\iota_E)$ is not the coproduct of $Se$ and $E$ in the category of semigroups.\end{proposition}

\begin{proof} The mapping $\varphi\colon S\to Se\times E$ in the external direct product $Se\times E$ of the homomorphic images $Se$ and $E$ of $S$, defined by $\varphi(s)=(se,\pi_E(s))$ for every $s\in S$, is a semigroup isomorphism, because $$\begin{array}{l}\varphi(ss')=(ss'e,\pi_E(ss'))=(ses'e,\pi_E(s'))= \\  \quad\quad =(se,\pi_E(s))(s'e,\pi_E(s'))=\varphi(s)\varphi(s'),\end{array}$$ and the inverse of $\varphi$ is the mapping $\varphi^{-1}\colon Se\times E\to S$ defined by $\varphi^{-1}(x,y)=xy$ for every $(x,y)\in Se\times E$.

In order to prove (1), fix any semigroup $T$ and semigroup morphisms $f_e\colon T\to Se$ and $f_E\colon T\to E$. Define a mapping $f\colon T\to S$ setting $f(t)=f_e(t)f_E(t)$ for every $t\in T$. Then $f$ is a semigroup morphism, because, for every $t,t'\in T$, we have that $$\begin{array}{l} f(t)f(t')=f_e(t)f_E(t)f_e(t')f_E(t')= \\ \quad\quad =f_e(t)f_e(t')f_E(t')=f_e(tt')f_E(tt')=f(tt').\end{array}$$ Also $r_ef=f_e$ and $\pi_E f=f_E$, because, for every $t\in T$, we have that $r_ef(t)=f_e(t)f_E(t)e=f_e(t)e=f_e(t)$ and $\pi_E f(t)=\pi_E(f_e(t)f_E(t))=f_E(t)$. 

As far as uniqueness of $f$ is concerned, let $g\colon T\to S$ be any other semigroup morphism such that $r_eg=f_e$ and $\pi_E g=f_E$. Then, for every $t\in T$ we have that $g(t)=\varphi^{-1}\varphi g(t)=\varphi^{-1}(g(t)e,\pi_E(g(t)))=\varphi^{-1}(r_eg(t),\pi_Eg(t))=\varphi^{-1}(f_e(t),f_E(t))=f_e(t)f_E(t)=f(t)$, so that $g=f$.

As far as (2) is concerned, let $(T,e_T)$ be a pointed right group and\linebreak  $f_e\colon Se\to T$ and $f_E\colon E\to T$ be two semigroup morphisms with $f_e(e)=f_E(e)=e_T$. Define a mapping $\psi\colon S\to T$ setting $\psi(s)=f_e(se)f_E(\pi_E(s))$ for every $s\in S$. Notice that all elements of $E$ are idempotent, so that every element of the image $f_E(E)$ is an idempotent element of $T$. By Lemma~\ref{1.2}, every element of $f_E(E)$ is a left identity for $T$. The mapping $\psi$ is a semigroup morphism, because for every $s,s'\in S$ we have that $$\begin{array}{l} \psi(ss')=f_e(ss'e)f_E(\pi_E(ss'))=\\ \quad\quad =f_e(ses'e)f_E(\pi_E(s'))=  f_e(se)f_e(s'e)f_E(\pi_E(s'))=\\ \quad\quad =f_e(se)f_E(\pi_E(s))f_e(s'e)f_E(\pi_E(s'))=\psi(s)\psi(s').\end{array}$$ Moreover, for every $s\in S$ we have that $$\psi\iota_e(se)=\psi(se)=f_e(se)f_E(e)=f_e(se)f_e(e)=f_e(se),$$ so that $\psi\iota_e=f_e$. Similarly, 
$$\psi\iota_E(e')=\psi(e')=f_e(e'e)f_E(\pi_E(e'))=f_e(e)f_E(e')=e_Tf_E(e')=f_E(e'),$$ hence $\psi\iota_E=f_E$.

For uniqueness, let $\psi'\colon S\to T$ be any other semigroup morphism such that $\psi'\iota_e=f_e$ and $\psi'\iota_E=f_E$. Then, for every $s\in S$ we have that $s=(se)\pi_E(s)$ with $se\in Se$ and $\pi_E(s)\in E$, so $$\psi'(s)=\psi'(se)\psi'(\pi_E(s))=\psi'\iota_e(se)\psi'\iota_E(\pi_E(s))=f_e(se)f_E(\pi_E(s))=\psi(s).$$ Therefore 
$\psi'=\psi$.

For (3), let $A=\{x,y\}$ be a set of two elements and $W$ the free semigroup freely generated by the set $A$, so that $W$ is the semigroup of all words of length $\ge1$ in the alphabet $A$ with respect to justapposition. Let $\simeq$ be the congruence on $W$ generated by the set of the two pairs $(x,xx)$ and $(y,yy)$. Set $T:=W/{\simeq}$, so that its two elements $[x]_\simeq$ and $[y]_\simeq$ are two distinct idempotents. Consider the constant semigroup morphisms $f_e\colon Se\to T$ and $f_E\colon E\to T$ defined by $f_e(se)=[x]_\simeq$ and $f_e(e')=[y]_\simeq$ for every $se\in Se$ and $e'\in E$. Then there is no mapping $\omega\colon S\to T$ such that $\omega\iota_e=f_e$ and $\omega\iota_E=f_E$, because $\omega(e)=\omega\iota_e(e)=f_e(e)=[x]_\simeq$  and $\omega(e)=\omega\iota_E(e)=f_E(e)=[y]_\simeq$, a contradiction. \end{proof}

The category of pointed right groups is a category with a zero object and, by Proposition~\ref{coproduct}, $(S,r_e,\pi_E,\iota_e,\iota_E)$ is a biproduct.

\begin{corollary}\label{coproduct'} If $S$ is a right group, $\pi_\sim\colon S\to S/{\sim}$ is defined by $\pi_E(s)=sE$ for every $s\in S$, and $\pi_E\colon S\to E$ is defined by ``$\pi_E(s)$ is the unique element of $S$ such that $s\pi_E(s)=s$'', then  $(S, \pi_{\sim},\pi_E)$ is the product of $S/{\sim}$ and $E$ in the category of semigroups.\end{corollary}

The proof follows from the isomorphism $\overline{r_e}\colon S/{\sim}\to Se$ and Proposition~\ref{coproduct}(1).

\smallskip

The difference between right groups and pointed right groups is not stated so explicitly in \cite{Kinyon}.

\begin{theorem}\label{2.7} There is a faithful, essentially surjective functor from the category of group actions on pointed sets to the category of pointed right groups.\end{theorem}

\begin{proof} Associate to any group action $(G,X, x_0, \varphi)$ on the pointed set $(X,x_0)$ the semigroup $(X\times G)_\varphi:=X\times G$ with the operation defined by $$(x,g)(x',g')=(\varphi_g(x'),gg')$$ for all $(x,g),(x',g')\in X\times G$. It is very easy to check that this semigroup $X\times G$ is a right group. It is pointed relatively to its idempotent element $(x_0,1_G)$.  Given any morphism $(G,X, x_0, \varphi)\to (H,Y,y_0,\psi)$ in the category of left group actions on pointed sets, we have that the morphism is a pair $(\Phi, f)$ with $\Phi\colon G\to H$ a group morphism, $f\colon X\to Y$ a mapping, $f(x_0)=y_0$ and $\psi_{\Phi(g)}(f(x))=f(\varphi_g(x))$ for every $g\in G$ and every $x\in X$. It is possible to associate to $(\Phi, f)\colon (G,X,  \varphi)\to (H,Y,\psi)$  the semigroup morphism $F(\Phi, f)=f\times \Phi\colon X\times G\to Y\times H$ defined by $(f\times\Phi)(x,g)=(f(x),\Phi(g))$ for all $x\in X$, $g\in G$. Clearly, the mapping $(\Phi, f)\to F(\Phi, f)=f\times \Phi$ is injective. That is, we have a faithful functor $F$ defined by $F(G,X, x_0,\varphi)=(X\times G, (x_0,1_G))$ and $F(\Phi, f)=f\times \Phi$.

Given any pointed right group $(S,e_0)$, associate to it the 4-tuple\linebreak $(Se_0, E(S), x_0,\mathds{1})$, where $\mathds{1}\colon Se_0\to\Sym_{E(S)}$ is the trivial action, for which $\mathds{1}_{se_0}\colon E(S)\to E(S)$ is the identity mapping of $E(S)$ for every $se_0\in Se_0$.
Then $F(Se_0, E(S), x_0,\mathds{1})=E(S)\times Se_0$, the semigroup with operation $$(x,se_0)(x',s'e_0)=(x',ss'e_0'),$$ so that $F(Se_0, E(S), x_0,\mathds{1})\cong S$. This proves that $F$ is essentially surjective.
 \end{proof}

\begin{theorem}\label{functor2} There is a faithful, essentially surjective functor from the category of group actions on non-empty sets to the category of right groups. \end{theorem}

\begin{proof} The proof is similar to that of Theorem~\ref{2.7}. Define a functor $F$ of the category of left group actions on non-empty sets to the category of right groups as follows. Associate to any left group action $(G,X, \varphi)$ on a non-empty set $X$ the semigroup $(X\times G)_\varphi:=X\times G$  with the operation defined by $(x,g)(x',g')=(\varphi_g(x'),gg')$ for all $(x,g),(x',g')\in X\times G$. It is very easy to check that $X\times G$ with this operation is a right group. Set $F(G,X, \varphi)=(X\times G)_\varphi$. Given any morphism $(G,X,  \varphi)\to (H,Y,\psi)$ in the category of left group actions on non-empty sets, we have that the morphism is a pair $(\Phi, f)$ with $\Phi\colon G\to H$ a group morphism, $f\colon X\to Y$ a mapping, and $\psi_{\Phi(g)}(f(x))=f(\varphi_g(x))$ for every $g\in G$ and every $x\in X$. It is possible to associate to $(\Phi, f)\colon (G,X,  \varphi)\to (H,Y,\psi)$  the semigroup morphism $F(\Phi, f)=f\times \Phi\colon (X\times G)_\varphi\to (Y\times H)_\psi$ defined by $(f\times\Phi)(x,g)=(f(x),\Phi(g))$ for all $x\in X$, $g\in G$.  In this way we get a faithful functor from the category of group actions on non-empty sets to the category of right groups. 

Given any right group $S$, consider the triplet $(S/{\sim}, E(S), \mathds{1})$, where $\sim$ is the congruence on $S$ defined by $x\sim y$ if $xe=ye$ for some idempotent element $e$ (see Remark~\ref{sim}; the congruence class of $x\in S$ modulo the congruence $\sim$ is $xE$), and $\mathds{1}\colon S/{\sim}\to \Aut_{\Set}(E(S))$ is the trivial group morphism. Then 
$$F(S/{\sim}, E(S), \mathds{1})=(E(S)\times S/{\sim})_{\mathds{1}}=E(S)\times S/{\sim}\cong S,$$
in view of Theorem \ref{1.1} and Remark \ref{sim}. This proves that $F$ is essentially surjective. \end{proof}

\begin{proposition}\label{right-group-iso-F}
Let $F$ be the functor defined in Theorem \ref{functor2}, let $G$ be a group and $X$ any nonempty set. Then, for every group morphism $\varphi:G\to \Aut_\Set(X)$, 	the right groups $F(G,X,\varphi)$ and $F(G,X,\mathds{1})$ are isomorphic, where $\mathds{1}$ is the trivial group morphism. 
\end{proposition}
\begin{proof}
Let $\eta:F(G,X,\mathds{1})\to F(G,X,\varphi)$ be the mapping defined by setting 
$
\eta(x,g):=(\varphi_g(x),g)
$ for all $x\in X$, $g\in G$. It is straightforward to see that $\eta$ is an isomorphism of right groups. 
\end{proof}
\begin{remark}
Observe that the functor $F$ defined in Theorem \ref{functor2} is not full. As a matter of fact, consider a group $G$ and a nonempty set $X$ in such a way there exists a nontrivial group morphism $\varphi:G\to \Aut_{\Set}(X)$. By Proposition, \ref{right-group-iso-F}, the right groups $F(G,X,\varphi)$ and $F(G,X,\mathds{1})$ are isomorphic. On the other hand, the group actions $(G,X,\varphi),(G,X,\mathds{1})$ are not isomorphic: indeed, if there exists an isomorphism $(\Phi,f):(G,X,\varphi)\to(G,X,\mathds{1})$, for some group automorphism $\Phi:G\to G$ and some bijection $f:X\to X$, then it would follow (from the definition of morphism in the category of group actions) that $f(x)=f(\varphi_g(x))$, for all $x\in X$, $g\in G$, and this would force (since $f$ is bijective) each $\varphi_g$ to be trivial, that is, $\varphi=\mathds{1}$, a contradiction. This shows that $F$ is not full. 

Similarly, it can be seen that the functor defined in Theorem \ref{2.7} is not full either. 
\end{remark}

\section{Homomorphisms of right groups}\label{subs}

The results in Subsection \ref{RG} and Section \ref{Crg} suggest to investigate semigroup morphisms $\varphi\colon S\to S'$ between two right groups $S,S'$. Let\linebreak $\Hom_\SGrp(S,S')$ denote the set of all such morphisms. Since images of idempotents via semigroup morphisms are idempotents, we have that $\varphi$ maps the set $E=E(S)$ of all idempotents of $S$ into the set $E'=E(S')$ of all idempotents of $S'$. Thus it is possible to consider the restriction $\varepsilon\colon E\to E'$ of $\varphi$ to $E$. Fix an element $e_0\in E$. Then $\varphi$ maps the group $Se_0$ to the group $S'\varepsilon(e_0)$, so that it is possible to consider the restriction $\varphi|_{Se_0}\colon Se_0\to S'\varepsilon(e_0)$, which is a group morphism.

\begin{proposition}\label{3.4} Let $S,S'$ be right groups. For every semigroup morphism $\varphi\colon S\to S'$ consider the triplet  $(\varepsilon\colon E\to E',e_0,\varphi|_{S{e_0}}\colon Se_0\to S'\varepsilon(e_0))$ described in the previous paragraph. Then:

{\rm (a)} The triplet $(\varepsilon\colon E\to E',e_0,\varphi|_{S{e_0}}\colon Se_0\to S'\varepsilon(e_0))$ completely determines the semigroup morphisms $\varphi\colon S\to S'$.

{\rm (b)} For every mapping $\varepsilon\colon E\to E'$, any element $e_0\in E$ and any group morphism $\psi\colon Se_0\to S'\varepsilon(e_0)$, the triplet $(\varepsilon,e_0,\psi)$ corresponds to a semigroup morphism $\varphi\colon S\to S'$, that is, there exists a semigroup morphism $\varphi\colon S\to S'$ such that $\varepsilon(e)=\varphi(e)$ for every $e\in E$ and $\psi(x)=\varphi(x)$ for every $x\in Se_0$.

{\rm (c)} Two such triplets $(\varepsilon_1,e_1,\psi_1), (\varepsilon_2,e_2,\psi_2)$ correspond to the same semigroup morphism $\varphi\colon S\to S'$ if and only if $\varepsilon_1=\varepsilon_2$ and the diagram \begin{equation}\xymatrix{
    Se_1 \ar[r]^{\psi_1} \ar[d]_{r_{e_2}} &S'\varepsilon_1(e_1)  \ar[d]^{r_{\varepsilon_2(e_2)}}  \\
    Se_2 \ar[r]^{\psi_2}  &S'\varepsilon_2(e_2)  }   \label{dia}\end{equation} commutes.
\end{proposition} 

\begin{proof} (a) For a given semigroup morphism $\varphi\colon S\to S'$ consider the triplet  $(\varepsilon\colon E\to E',e_0,\varphi|_{S{e_0}}\colon Se_0\to S'\varepsilon(e_0))$. If we consider the pointed right groups $(S,e_0)$ and $(S',\varepsilon(e_0))$,  $\varphi\colon S\to S'$ becomes a morphism  $\varphi\colon (S,e_0)\to (S',\varepsilon(e_0))$ of pointed right groups. In view of Proposition~\ref{coproduct}((1) and (2)), $S$ is both a product and a coproduct of $Se_0$ and $E$. Similarly, $S'$ is both a product and a coproduct of $S'\varepsilon(e_0)$ and $E'$.  Now every element $s\in S$ can be written in a unique way as a product of the element $se_0$ of $Se_0$ and the element $\pi_E(s)$ of $E$. Therefore \begin{equation}\varphi(s)=\varphi((se_0)(\pi_E(s)))=\varphi(se_0)\varphi(\pi_E(s))=\varphi|_{S{e_0}}(se_0)\varepsilon(\pi_E(s)).\label{equat}\end{equation} Hence the triplet $(\varepsilon\colon E\to E',e_0,\varphi|_{S{e_0}}\colon Se_0\to S'\varepsilon(e_0))$ completely determines $\varphi\colon S\to S'$. 

(b) Given a triplet  $(\varepsilon,e_0,\psi)$ as in (b), define $\varphi\colon S\to S'$ setting, for every $s\in S$, $\varphi(s)=\psi(se_0)\varepsilon(\pi_E(s))$. Then  $\varphi$ is a semigroup morphism, because $$\begin{array}{l}\varphi(s_1)\varphi(s_2)=\psi(s_1e_0)\varepsilon(\pi_E(s_1))\psi(s_2e_0)\varepsilon(\pi_E(s_2))=\\ \quad\quad = \psi(s_1e_0)\psi(s_2e_0)\varepsilon(\pi_E(s_2))=\psi(s_1s_2e_0)\varepsilon(\pi_E(s_1)\pi_E(s_2))=\\ \quad\quad\quad\quad  = \varphi(s_1s_2).\end{array}$$

(c) Two triplets $(\varepsilon_1,e_1,\psi_1), (\varepsilon_2,e_2,\psi_2)$ correspond to the same semigroup morphism $\varphi\colon S\to S'$ if and only if $$\varphi(s)=\psi_1(se_1)\varepsilon_1(\pi_E(s))=\psi_2(se_2)\varepsilon_2(\pi_E(s))$$  for every $s\in S$. Hence if $(\varepsilon_1,e_1,\psi_1), (\varepsilon_2,e_2,\psi_2)$ correspond to the same semigroup morphism $\varphi\colon S\to S'$, then $\varepsilon_1=\varepsilon_2$, because they are both the restriction of $\varphi$ to $E$. Moreover Diagram (\ref{dia}) commutes because, for every $s\in S$, $r_{\varepsilon_2(e_2)}\psi_1(se_1)=r_{\varphi(e_2)}\varphi(se_1)=\varphi(se_1)\varphi(e_2)=\varphi(se_1e_2)=\varphi(se_2)$ and $\psi_2r_{e_2}(se_1)=\psi_2(se_1e_2)=\psi_2(se_2)=\varphi(se_2)$. Therefore $r_{\varepsilon_2(e_2)}\psi_1=\psi_2r_{e_2}$, and the diagram commutes.

For the converse, we must prove that if $\varepsilon_1=\varepsilon_2$ and Diagram~(\ref{dia})  commutes, then $\psi_1(se_1)\varepsilon_1(\pi_E(s))=\psi_2(se_2)\varepsilon_2(\pi_E(s))$ for every $s\in S$. Now $\varepsilon_1=\varepsilon_2$ and the commutativity of the diagram imply that $r_{\varepsilon_1(e_2)}\psi_1(se_1)=\psi_2r_{e_2}(se_1)$, that is, $\psi_1(se_1)\varepsilon_1(e_2)=\psi_2(se_2)$ for every $s\in S$. Multiplying on the right by $\varepsilon_1(\pi_E(s))=\varepsilon_2(\pi_E(s))$, we get that $\psi_1(se_1)\varepsilon_1(\pi_E(s))=\psi_2(se_2)\varepsilon_2(\pi_E(s))$, as desired. \end{proof}

Notice that the vertical arrows in Diagram (\ref{dia})  are group morphisms. More generally, for any two elements $e,f\in E=E(S)$, where $S$ is right group, the mapping $r_f\colon Se\to Sf$, $r_f(s)=sf$ for all $s\in Se$, is a group isomorphism, because $$r_f(se)r_f(s'e)=sefs'ef=ss'ef=ses'ef=r_f((se)(s'e)).$$ Its inverse is the mapping $r_e\colon Sf\to Se$. This is the reason why the restriction $\varphi|_{Se_0}$ determines the behavior of $\varphi$ on all the blocks of the partition $\{\,Se\mid e\in E\,\}$ of~$S$.

\medskip

From Proposition \ref{3.4} it can be easily shown that:

\begin{theorem}\label{equiv2} The category $\RGrp_*$ of pointed right groups is equivalent to the product category $\Set_*\times\Grp$ of the category $\Set_*$ of pointed sets and the category $\Grp$ of groups.\end{theorem}

\subsection{Right inverses $\varphi\colon S/{\sim}\to S$ to the canonical projection $\pi\colon S\to S/{\sim}$.}
In order to try to limit as much as possible the need of introducing the artificial concept of pointed right group, it is convenient to consider the right inverses of the canonical projection $\pi\colon S\to S/{\sim}$. Here $S$ is a right group and $\sim$ is the semigroup congruence on $S$ considered in Remark~\ref{sim}. For any $e\in E$, $\sim$ is the kernel of the semigroup morphism $r_e\colon S\to S$ defined by $r_e(x)=xe$ for every $x\in S$. The congruence class modulo $\sim$ of every element  $x\in S$ is $xE=\{\,xe\mid e\in E\,\}$.  For every right group $S$, the quotient semigroup $S/{\sim}$ is a group. 

If $\psi\colon S\to S'$ is any semigroup morphism between two right groups $S$ and $S'$, then $\psi$ induces a group morphism $\widetilde{\psi}\colon S/{\sim}\to S'/{\sim}$. To see it, fix an element $e_0\in E$. If $s,t\in S$ and $s\sim t$, then $se_0=te_0$, so $\psi(s)\psi(e_0)=\psi(t)\psi(e_0)$. Then $\psi(s)\sim\psi(t)$. Thus $\widetilde{\psi}$ is a well defined mapping. Thus $S\mapsto S/{\sim}$, $\psi\mapsto \widetilde{\psi}$, is a functor $\RGrp\to\Grp$.

\begin{proposition}\label{inv} For a right group $S$, there is a one-to-one correspondence between the set of the semigroup morphisms that are right inverses of the canonical projection $\pi\colon S\to S/{\sim}$ and the set $E(S)$. If $e_0\in E(S)$, the right inverse homomorphism of $\pi$ corresponding to $e_0$ is the semigroup morphism $\overline{r_{e_0}}\colon S/{\sim}\to S$ induced by right multiplication $r_{e_0}\colon S\to S$ by $e_0$.\end{proposition}

\begin{proof}
If $\varphi\colon S/{\sim}\to S$ is any semigroup morphism such that $\pi\varphi=\id_{S/{\sim}}$, then $\varphi$ must map the identity $E$ of the group $S/{\sim}$ to an idempotent element $e_0$ of $S$, that is, to an element $e_0\in E:=E(S)$. We have the direct-product decomposition $S=Se_0\times E$ and, correspondingly, the trivial direct-product decomposition $S/{\sim}=(S/{\sim})E\times \{E\}$. Let us show that the restriction $\pi|_{Se_0}\colon Se_0\to S/{\sim}=(S/{\sim})E$ of the canonical projection $\pi\colon S\to S/{\sim}$ is a (semi)group isomorphism. The mapping $\pi|_{Se_0}$ is defined by $\pi|_{Se_0}(se_0)=se_0E=sE$ for every $se_0\in Se_0$. It is injective because if $s,s'\in S$ and $sE=s'E$, then, multiplying by $e_0$ on the right we get that $se_0=s'e_0$. It is surjective, because if $sE\in S/{\sim}$, then, multiplying $s$ by $e_0$ on the right, we see that $\pi|_{Se_0}(se_0)=sE$. Therefore $\pi|_{Se_0}$ is an isomorphism.

  In the notations of Proposition~\ref{3.4}, the triplet  $$(\varepsilon\colon E(S/{\sim})\to E(S),E,\varphi|_{S/{\sim}}\colon S/{\sim}\to Se_0)$$ corresponding to the semigroup morphism $\varphi\colon S/{\sim}\to S$ is such that\linebreak $\varepsilon\colon E(S/{\sim})\to E(S)$ maps the unique element $E$ of $E(S/{\sim})$ to $e_0$. Also, $\pi\varphi=\id_{S/{\sim}}$ implies $\pi|_{Se_0}\varphi|^{Se_0}=\id_{S/{\sim}}$. Since $\pi|_{Se_0}\colon Se_0\to S/{\sim}=(S/{\sim})E$ is an isomorphism, it follows that $\pi|_{Se_0}\colon Se_0\to S/{\sim}$ and $\varphi|^{Se_0}\colon S/{\sim}\to Se_0$ are mutually inverse group isomorphisms. Thus $\varphi|^{Se_0}\colon S/{\sim}\to Se_0$ maps any element $sE$ of $S/{\sim}$ to $se_0$. Now $r_{e_0}\colon S\to S$ induces an injective homomorphism $\overline{r_{e_0}}\colon S/{\sim}\to S$, and it is easily seen that $\overline{r_{e_0}}$ also corresponds to the same triplet $(\varepsilon,E,\varphi|_{S/{\sim}})$ as $\varphi$. \end{proof}

\begin{proposition} {\rm (a)} If $(S,\cdot)$ is a right group and $e,f\in E(S)$, then $(S,\cdot,e)$ and $(S,\cdot,f)$ are isomorphic pointed right groups.

{\rm (b)} The forgetful functor $\RGrp_*\to \RGrp$ that associates to each pointed right group $(S,\cdot,e)$ the right group $(S,\cdot)$ is a faithful, essentially surjective functor.\end{proposition}

\begin{proof} (a) Let $\varepsilon\colon E(S)\to E(S)$ be any bijection that maps $e$ to $f$. Let $r_f\colon Se\to Sf$ be the group isomorphism given by right multiplication by $f$. The triplet $(\varepsilon, e, r_f)$ corresponds to an isomorphism $(S,\cdot,e)\to(S,\cdot,f)$. The proof of (b) is easy.\end{proof}

\begin{theorem}\label{equiv1} The category of right groups is equivalent to the product category of the category $\Set_{\ne\emptyset}$ of non-empty sets and the category $\Grp$ of groups.\end{theorem}

\begin{proof} The category equivalence is the functor $$F\colon\RGrp\to\Set_{\ne\emptyset}\times\Grp$$ that associates to every right group $S$ the pair $(E(S),S/{\sim})$, where $\sim$ is the kernel of any right multiplication $r_e\colon S\to S$ (Remark~\ref{sim}). The functor $F$ associates to every right group morphism $f\colon S\to S'$ the pair of morphisms $(f|_E\colon E(S)\to E(S'), \widetilde{f}\colon S/{\sim}\to S'/{\sim})$, where $f|_E$ is the restriction of $f$ to $E(S)$, and $\widetilde{f}$ is the group morphism of $S/{\sim}$ into $S'/{\sim}$ induced by $f$. In order to prove that $F$ is full and faithful, we must prove that the mapping \begin{equation}\begin{array}{c}\Hom_{\RGrp}(S,S')\to \Hom_{\Set}(E(S),E(S'))\times\Hom_{\Grp}(S/{\sim},S'\sim), \\ f\mapsto (f|_{E(S)}, \widetilde{f})\end{array}\label{biiez} \end{equation} is a bijection for all right groups $S,S'$. Fix any element $e_0\in E(S)$. Then right multiplication $r_{e_0}\colon S\to Se_0$ is a surjective semigroup morphism with kernel $\sim$, hence it induces a (semi)group isomorphism $\overline{r_{e_0}}\colon S/{\sim}\to Se_0$. Since $f$ is a semigroup morphism, we have the commutative diagram \[ \xymatrix{
    S \ar[r]^{r_{e_0}} \ar[d]_{f} &Se_0 \ar[d]^{f|_{Se_0}}  \\
    S' \ar[r]_{r_{f(e_0)}}  &S'f(e_0),  }   \] so that $\widetilde{f}=(\overline{r_{f(e_0)}})^{-1}f|_{Se_0}\overline{r_{e_0}}$. Hence, if $f$ corresponds to the triplet $$(f|_E, e_0,f|_{Se_0})$$ in the sense of Proposition~\ref{3.4}, then the mapping in (\ref{biiez}) associates to $f$ the pair $(f|_{E(S)}, \widetilde{f}=(\overline{r_{f(e_0)}})^{-1}f|_{Se_0}\overline{r_{e_0}})$. The proof that the mapping in (\ref{biiez}) is bijective follows therefore from Proposition~\ref{3.4}.
    
    Finally, the functor $F$ is essentially surjective. In order to see it, associate to any set $X$ and any group $G$, the direct product of $G$ and the right zero semigroup $(X,\pi_2)$.
\end{proof}

In view of Theorem~\ref{equiv1}, it is natural to define as {\em kernel} of a morphism $\psi\colon S\to S'$ between two right groups the pair $(\sim_{\psi|_E},K)$, where $\sim_{\psi|_E}$ is the kernel of the mapping $\psi|_E\colon E\to E'$ (it is a partition of the set $E$) and $K$ is the kernel of the group morphism $\widetilde{\psi}\colon S/{\sim}\to S'/{\sim})$ (it is a normal subgroup of the group $S/{\sim}$). 

The inverse image of an element $s'\in S'$ via the semigroup morphism $\psi\colon S\to S'$ can be computed as follows. Given $s'\in S'$, we have that $s'E'\in S'/{\sim}$ and $\pi_{E'}(s')\in E'$. Assume that $s'E'\in\Imm(\widetilde{\psi})$ and $\pi_{E'}(s')\in \Imm(\psi|_E)$. Then there exist $s_1\in S$ such that $\psi(s_1)\in s'E$ and $s_2\in E(S)$ such that $\psi(s_2)=\pi_{E'}(s')$.
In this case, we have $$\begin{array}{l}\psi^{-1}(s')=(\widetilde{\psi})^{-1}(s'E')\cdot(\psi|_E)^{-1}(\pi_{E'}(s'))= \\ \qquad\qquad\qquad =s_1E\cdot K\cdot[s_2]_{\sim_{\psi|_E}}=s_1K\cdot[s_2]_{\sim_{\psi|_E}}.\end{array}$$ In the other case, that either  $s'E'\notin\Imm(\widetilde{\psi})$ or $\pi_{E'}(s')\notin \Imm(\psi|_E)$, we have that $\psi^{-1}(s')=\emptyset$. 

\section{Pretorsion theory for right groups}\label{pretorsion}

We now briefly recall the notions developed in \cite{FF} and \cite{fa-fi-gr} about pretorsion theories in arbitrary categories.
Let $\mathsf{C}$ be a category and $\mathsf{Z}$ be a non-empty class of objects of $\mathsf{C}$. For every pair $A,A'$ of objects of $\mathsf{C}$, we indicate by $\Triv_{\mathsf{Z}}(A, B)$ the set of  all morphisms in $\mathsf{C}$ that factor through an object of $\mathsf{Z}$. We call these morphisms {\em $\mathsf{Z}$-trivial}.

If $f\colon A\to A'$ is a morphism in $\mathsf{C}$, a morphism $\varepsilon\colon X\to A$ in $\mathsf{C} $ is a \emph{$\mathsf{Z}$-prekernel} of $f$ if: 
\begin{enumerate}
	\item $f\varepsilon$ is a $\mathsf{Z}$-trivial morphism.
	\item If $\lambda \colon Y\to A$ is any morphism in $\mathsf{C}$ for which $f\lambda$ is $\mathsf{Z}$-trivial, then there exists a unique morphism $\lambda'\colon Y\to X$ in $\mathsf{C}$ such that $\lambda=\varepsilon\lambda'$. 
\end{enumerate}
Dually, a \emph{$\mathsf{Z}$-precokernel} of $f$ is a morphism $\eta\colon A'\to X$ such that:
\begin{enumerate}
	\item $\eta f$ is a $\mathsf{Z}$-trivial morphism.
	\item If $\mu\colon A'\to Y$  is any morphism in $\mathsf{C}$ for which $\mu f$ is $\mathsf{Z}$-trivial, then there exists a unique morphism $\mu'\colon X\to Y$ with $\mu=\mu' \eta$.
\end{enumerate}

If $f\colon A\to B$ and $g\colon B\to C$ are morphisms in $\mathsf{C}$, we say that $$\xymatrix{
	A \ar[r]^f &  B \ar[r]^g &  C}$$ is a \emph{short $\mathsf{Z}$-preexact sequence} in $\mathsf{C}$ if $f$ is a $\mathsf{Z}$-prekernel of $g$ and $g$ is a $\mathsf{Z}$-precokernel of $f$.

  \begin{definition}\label{pretorsion-theory-def} {\rm Let $\mathsf{C}$ be a category, and $\mathsf{T},\mathsf{F}$ be two replete (that is, closed under isomorphism) full subcategories of $\mathsf{C}$. Set $\mathsf{Z}:=\mathsf{T}\cap\mathsf{F}$. The pair $(\mathsf{T},\mathsf{F})$ is a {\em pretorsion theory}  in the category $\mathsf{C}$ if the following properties hold.
  	
	(1) $\Hom_{\mathsf{C}}(T,F)=\Triv_{\mathsf{Z}}(T, F)$  for every object $T\in\mathsf{T}$, $F\in\mathsf{F}$.
	
	  (2) For every object $B$ of $\mathsf{C}$ there is a short $\mathsf{Z}$-preexact sequence $$\xymatrix{
	A \ar[r]^f &  B \ar[r]^g &  C}$$ with $A\in\mathsf{T}$ and $C\in\mathsf{F}$.}\end{definition}

	Now let $\mathsf{C}:=\Rg$ be the category of right groups, let $\mathsf{T}:=\Rzs$ (resp., $\mathsf{F}:=\Grp$) be the full subcategories of $\Rg$ consisting of right zero semigroups (resp., groups),  and $\mathsf{Z}:=\mathsf{T}\cap\mathsf{F}$. Clearly $\mathsf Z$ consists of all (semi)groups of order $1$, i.e., the terminal objects of $\Rg$. Thus $\mathsf Z$-trivial morphisms $S\to S'$ are the semigroup morphisms whose image is a singleton, that is, they are exactly the constant morphisms $S\to S'$, that is, the mappings $f\colon S\to S'$ for which there exists an element $e_0'\in E(S')$ for which $f(s)=e_0'$ for every $s\in S$.
	
	\medskip

	Our first goal is to characterize morphisms in $\RGrp$ that admit a $\mathsf Z$-pre\-ker\-nel. 
	
	\begin{lemma}\label{Z-prekernel-charact}
		Let $f\colon S\to S'$ be a morphism in $\RGrp$ and $E:=E(S)$. Then $f$ has a $\mathsf Z$-prekernel in $\Rg$  if and only if $f(E)$ is a singleton. 
		Moreover, if $f$ has a $\mathsf Z$-prekernel in $\Rg$ and $f(E)=\{g_0\}$, then $K:=f^{-1}(\{g_0\})$ is a right subgroup of $S$, and the inclusion morphism $i\colon K\to S$ is a $\mathsf Z$-prekernel of $f$. 
	\end{lemma}
	
	\begin{proof}
		Assume that $f(E)=\{g_0\}$, so that, in particular, $g_0$ must be an idempotent of $S'$. Then $K:=f^{-1}(\{g_0\})$ is a subsemigroup of $S$. Now, let $a,b\in K$ and let $x\in S$ be the unique element such that $ax=b$ (Theorem~\ref{1.1}(a)). Then		$$g_0^2=g_0=f(b)=f(a)f(x)=g_0f(x),$$
		and thus $f(x)=g_0$, because $S$ is left cancelative. This proves that $x\in K$, and thus $K$ is a right subgroup of $S$. Let $i\colon K\to S$ be the inclusion.  By construction, the composition $fi$ is $\mathsf{Z}$-trivial. Consider now any morphism $\lambda\colon Y\to S$ of right groups such that $f\lambda$ is $\mathsf{Z}$-trivial, that is, there is an idempotent $g_1\in S'$ such that $f(\lambda(Y))=\{g_1\}$. Since $Y$ has an idempotent $e_0$ and $\lambda(e_0)\in E$, it immediately follows that $g_1=g_0$, proving that $\lambda(Y)\subseteq K$. Hence the mapping $\lambda'\colon Y\to K$, $y\mapsto \lambda(y)$, is the unique morphism in $\Rg$ such that $\lambda=i\lambda'$. This proves that $i$ is a $\mathsf Z$-prekernel of $f$. 
		
			Conversely, assume that $f$ has a $\mathsf Z$-prekernel $j\colon L\to S$. In particular, $fj$ is $\mathsf{Z}$-trivial, that is, $f(j(L))=\{g_0\}$ for some idempotent $g_0\in S'$. Since $L$ has an idempotent $l_0$ and $j(l_0)\in E$, it follows that $g_0\in f(E)$. Consider now any element $e\in E$. Then the inclusion $\iota\colon \{e\}\to S$ is a morphism in $\Rg$ and the composition $f\iota$ is a constant morphism of right groups, and thus it is $\mathsf{Z}$-trivial. Since $j$ is a $\mathsf Z$-prekernel of $f$, there is a unique morphism $\iota_0\colon \{e\}\to L$ such that $\iota=j\iota_0$. It follows that $e\in j(L)$ and thus $f(j(L))=\{g_0\}$ implies that $f(e)=g_0$. This proves that $f(E)=\{g_0\}$. 
	\end{proof}
	
	\begin{lemma}\label{pretorsion-cruciale}
		Let $\mu\colon S\to T$ be a morphism of right groups such that $\mu(E(S))$ is a singleton. Then the kernel $\sim$ of the canonical projection $\pi_G\colon S\to S/{\sim}$ is contained in the kernel of $\mu$. In particular, there exists a unique morphism $\mu'\colon S/{\sim}\to T$ such that $\mu=\mu'\pi_G$. 
	\end{lemma}
	\begin{proof} Fix any idempotent element $e_0$ of $S$. As we saw in the proof of Proposition~\ref{3.4}(a) (Identity~\ref{equat}), for every morphism $\mu\colon S\to T$ of right groups we have $\mu(s)=\mu|_{Se_0}(se_0)\mu|_E(\pi_E(s))$ for every $s\in S$. Assume that\linebreak $\mu(E(S))$ is a singleton, and let
	$t$ be the unique element of $\mu(E(S))$. In order to show that $\sim$ is contained in the kernel of $\mu$, let $s,s'$ be two elements of $S$ such that $s\sim s'$. Then $se_0=s'e_0$ (Remark~\ref{sim}) and $\mu|_E(\pi_E(s))=t=\mu|_E(\pi_E(s'))$, so that $\mu(s)=\mu(s')$. This proves that $\sim$ is contained in the kernel of $\mu$. From this it follows that $\mu\colon S\to T$ induces a unique semigroup morphism $\mu'\colon S/{\sim}\to T$, i.e., there is a unique morphism $\mu'\colon S/{\sim}\to T$ such that 
	$\mu=\mu'\pi_G$. 	
\end{proof}
	
	We have alreay remarked that the full subcategory $\mathsf{T}=\Rzs$ of $\Rg$ consisting of right zero semigroups  is isomorphic to the category $\Set$.
	
	\begin{theorem}
		The pair $(\Rzs,\Grp)$ is a pretorsion theory for $\Rg$. 
	\end{theorem}
	\begin{proof}
		Set $\mathsf Z=\Rzs\cap \Grp$. Consider a right zero semigroup $D$, a group $G$ with identity 1, and a morphism $\varphi\colon D\to G$ in $\Rg$. If $D$ is empty then  $\varphi$ is obviously the empty mapping and it factors as the composition of the empty mapping $D\to \{1\}$ and the inclusion $\{1\}\to G$, that is, $\varphi$ is $\mathsf Z$-trivial. Now suppose  that $D\neq\emptyset$. Since images of idempotents are idempotents, all elements of $D$ are idempotents, and the only idempotent in the group $G$ is  $1$, we immediately have $\varphi(D)=1$. Thus $\varphi$ is $\mathsf Z$-trivial.
		
		Now, let $S$ be any right group,  let $E:=E(S)$ be the set of all idempotents of $S$, consider the group $G:=S/{\sim}$ and the canonical projection $\pi_G\colon S\to G$, defined by $s\mapsto sE$, for every $s\in S$. From Lemma~\ref{Z-prekernel-charact}
		applied to the  morphism $\pi_G\colon S\to G$, we know that  $\pi_G$ has a $\mathsf Z$-prekernel in $\Rg$,  because $\pi_G(E)$ is the singleton $\{E\}$ (it is the trivial subgroup of the group $G=S/{\sim}$). Moreover, $E=\pi_G^{-1}(\{E\})$ and the inclusion morphism $i\colon E\to S$ is a $\mathsf Z$-prekernel of $f$ (Lemma~\ref{Z-prekernel-charact}). 

 It remains to show that $\pi$ is a $\mathsf Z$-precokernel of $i$. Let $\mu\colon S\to T$ be any morphism in $\Rg$ such that $\mu i$ is $\mathsf{Z}$-trivial. This means that $\mu(E)$ is a singleton. Hence we can apply Lemma~\ref{pretorsion-cruciale}, getting that there exists a unique morphism $\mu'\colon S/{\sim}\to T$ such that $\mu=\mu'\pi_G$. The mapping $\mu'$ is defined by $sE\mapsto \mu(s)$. This allows us to conclude.
	\end{proof}

\end{document}